\documentclass[12pt,oneside]{amsproc}
\usepackage[T2A]{fontenc}
\usepackage[cp1251]{inputenc}
\usepackage[english,russian]{babel}
\usepackage{amssymb}
\usepackage{amsthm}
   \usepackage{verbatim}

\theoremstyle{plain}

\newtheorem{theorem}{Теорема}
\newtheorem{proposition}{Предложение}

\newtheorem{cor}{Следствие}
\newtheorem{definition}{Определение}

\theoremstyle{definition}
\newtheorem{remark}{Замечание}

\title[О произведении $l_{s,r}$-ядерных операторов]
{О произведении $l_{s,r}$-ядерных и близких к ним операторов}
\author{О.И. Рейнов}
\address{Санк-Петербургский государственный университет
}
\email{orein51@mail.ru}

\thanks{${ }$ Key words:  ядерный оператор, класс Шаттена, пространство Лоренца, факторизация, гильбертово пространство}

\begin{document}

\begin{abstract}
Цель статьи --- исследовать возможности факторизации различного типа
ядерных операторов
через гильбертовы пространства и применить получаемые
результаты к задачам о распределении
собственных чисел операторов из соответствующих классов.
\end{abstract}


  \maketitle

 \medskip

\section{Введение}

 Пожалуй, впервые задача о распределении собственных чисел ядерных операторов
появилась (неявно) в 1909 г. статье И. Шура \cite{Schur}.
Доказанное там неравенство для собственных чисел интегральных операторов в $L_2(a,b)$
с квадратично суммируемым ядром теперь известно как неравенство Шура (из него следует, что
собственные числа этих операторов лежат в $l_2).$ 
Заметим, что в случае, когда ядра непрерывны, эти интегральные операторы являются ядерными
в $C[a,b]$.

 В 1915 г. Т. Лалеско \cite{Lal}, обобщая теорему Шура, рассмотрел  интегральные операторы ,
представляющие собой
суперпозиции двух операторов Гильберта-Шмидта (такие операторы являются ядерными в $L_2(a,b)$),
установив, что эти операторы
имеют абсолютно суммируемую последовательности собственных чисел. В работе \cite{Lal} Е. Лалеско
не указывал явно пространство, в котором действуют его операторы, но если считать, например,
что ядра операторов непрерывны и сами операторы заданы в пространстве  $C[a,b],$
то они представят собой первые примеры произведений двух ядерных операторов. Полученная же
им теорема тогда является первой теоремой о том, что произведение двух ядерных операторов 
имеет абсолютно суммируемую последовательность собственных чисел (что потом, в 1955 г., в абстрактной форме
докажет А. Гротендик; см. ниже).

В 1916 г.  Т. Карлеман \cite{Ca16}
привел первый пример интегрального оператора с непрерывным ядром в  $C[0,2\pi]$
(это ядерный оператор), собственные числа которого лежат в $l_2\setminus \cup_{r<2} l_r.$
Этим была установлена точность теоремы Шура (если рассматривать ее как теорему об
интегральном операторе с непрерывным ядром).

Абстрактное понятие ядерного (более того, $s$-ядерного) оператора было введено в рассмотрение лишь в 1955 г.
А. Гротендиком \cite{Groth} (после известных работ Шаттена, фон Ноймана и др.). Им были получены
основные на то время результаты о распределении собственных чисел $s$-ядерных операторов.

\smallskip

После этой фундаментальной работы А. Гротндика над проблемой распределения собственных чисел как ядерных,
так и близких к ним операторов работало (и продолжает работать) огромное число математиков.
Невозможно перечислить всех основных авторов. В сравнительно близкий к исследованиям А. Гротендика
период этим серьезно занимались такие специалисты, как
В. Б. Лидский, A. Pietsch, А.С. Маркус и В.И. Мацаев, H. K\"onig, B. Maurey, W.B. Johnson и многие другие.
Соответствующие ссылки можно найти, например, в монографиях \cite{EigPie} и \cite{PiOP}.

Следует отметить фундаментальную работу \cite{WBJ79}, в которой, помимо получения большого числа важных результатов,
впервые был рассмотрен вопрос о распределении собственных чисел произведений
нескольких операторов в банаховых пространствах , принадлежащих различным операторным идеалам
(таких как идеалы абсолютно суммирующих операторов).

Эта заметка, как и предыдущие две \cite{FaRei, ReiMZ} возникла благодаря следующему вопросу Б.С. Митягина, 
заданному в 2014 г. на конференции, посвященной памяти А. Пелчинского, в
Бедлево (Польша):
верно ли, что произведение двух ядерных операторов в банаховых пространствах факторизуется через
ядерный оператор в гильбертовом пространстве?

В  \cite{FaRei}, используя пример Карлемана,
 мы показали, что ответ отрицателен. Там же были приведены точные результаты о факторизации
произведений $s$-ядерных операторов в банаховых пространствах через операторы из классов Шаттена \cite{SchatTensors}. 
Затем, в работе \cite{ReiMZ}, были получены их конечномерные аналоги и, в частности,
приведены доказательства анонсированных ранее утверждений.

Здесь мы исследуем более общие вопросы.
Именно, при каких (желательно, {\it точных})\, значениях параметров $p,q$ $p,q\in (0,+\infty],$
произведения нескольких так называемых $(s,r)$-ядерных (и близких к ним) операторов в банаховых пространствах факторизуется через
операторы в гильбертовом пространстве, принадлежащие классам Лоренца-Шаттена $S_{p,q}.$
Результаты применяются к некоторым задачам о распределении собственных чисел.

\section{Предварительные сведения}
Мы придерживаемся терминологии монографии \cite{PiOP}..  
Везде далее через $X, Y, \dots$ обозначаются банаховы пространства,
$L(X, Y)$ --- банахово пространство всех
линейных непрерывных операторов из $X$  в $Y.$ 
Для банахова сопряженного к пространству $X$ используется обозначение $X^*.$
Если $x\in X$ и $x'\in X^*,$ то используем обозначение $\langle x',x\rangle$ для $x'(x).$
Элементы пространств $X, X^*, Y$ и т.д. будут обозначаться через $x, x', y$ и т.д..
 Обозначения $c_0, l_p, L_p (0<p\le\infty)$ стандартны.

Пространство 
Лоренца $l_{p,q}$ $(0<p<\infty, 0\le\infty)$
состоит из последовательностей $\alpha:=(\alpha_n)\in c_0,$ для которых
$$
||\alpha||_{p,q}:= \left(\sum_{n\in\mathbb N} \alpha_n^{*q} n^{q/p-1}\right)^{1/q}<+\infty\
\text{при}\ q<\infty \ \text{и}\
||\alpha||_{p\infty}:= \sup_{n\in\mathbb N} \alpha^*_n n^{1/p}<+\infty,
$$
где $(\alpha^*_n)$ есть неубывающая перестановка последовательности $\alpha,$
$n$-й элемент $\alpha^*_n$ которой определяется так:
$$
\alpha^*_n:= \inf_{|J|<n} \sup_{j\notin J} |\alpha_j|.
$$
С указанными квазинормами пространства $l_{p,q}$ являются полными квазинормированными пространствами.
При $p=q<\infty$ получаем пространства $l_p$ (с квазинормой $||\cdot||_p).$ Естественно считать, что $l_{\infty\infty}=l_\infty$
(с квазинормой $||\cdot||_\infty).$
Отметим, что 
$l_{p,q_1} \underset{\neq}\subset l_{p,q_2}$ для  $q_1<q_2$ и
$l_{p_1,q_1}\underset{\neq}\subset l_{p_2,q_2}$для  $p_1<p_2 $ и для всех $q_1, q_2.$

 Оператор $T:X\to Y$ называется  {\it $s$-ядерным}\ $(0<s\le1,$ см., например, \cite{PiOP}), 
      если он представим в виде
   $ Tx=\sum_{k=1}^\infty \langle x'_k,x\rangle y_k$
для $x\in X,$ где  $$(x'_k)\subset X^*, (y_k)\subset Y,\, \sum_k ||x'_k||^s\,||y_k||^s<\infty.$$
 Мы используем обозначение
 $N_s(X,Y)$ для линейного пространства всех таких операторов
 и $\nu_s(T)$ для соответствующей квазинормы
$\inf   (\sum_k ||x'_k||^s\,||y_k||^s)^{1/s}.$
 В случае, когда $s=1,$
 эти операторы называют просто {\it ядерными}.
  Оператор $T:X\to Y$ называется   {\it $(s,r)$-ядерным}\ $(0<s<1, 0<r\le\infty$ или $s=1, 0<r\le1;$ см., например, \cite{HiPi}), 
    если он может быть представлен в виде
   $ Tx=\sum_{n=1}^\infty a_n \langle x'_n,x\rangle y_n$
для $x\in X,$ где  $(x'_n)\subset X^*,$ $ (y_n)\subset Y,$  $||x'_n||, ||y_n||\le1,$ $(a_n)\in l_{s,r}.$
Отметим, что мы можем (и будем) предполагать, что $||x'_n||= ||y_n||=1$ для всех $n$ и что
последовательность $(a_n)$  неотрицательная и убывающая..
Мы используем обозначения
 $N_{s,r}(X,Y)$ для векторного пространства всех таких операторов
и $\nu_{s,r}(T)$ для соответствующей квазинормы
$\inf   ||(a_n)||_{l_{s,r}}.$
В случае, когда $s=r=1$ эти операторы называются {\it ядерными}.

Ниже мы рассматриваем только $(s,r)$-ядерные операторы для показателей,
удовлетворяющих неравенствам $0<r\le s\le1.$
  Каждый $(s,r)$-ядерный оператор $T: X\to Y$ допускает факторизацию следующего вида:
 \begin{equation}\label{Eq1}
  T:\, X\overset{W}\to l_\infty\overset{\Delta}\to l_1\overset{V}\to Y,
   \end{equation}
  где $||V||=||W||=1$ и $\Delta$ --- диагональный оператор с диагональю $(d_n)\in l_{s,r}.$
  Действительно, достаточно положить
  $Wx:= (\langle x'_k,x\rangle),$ $V(\alpha_n):= \sum \alpha_n y_n$ и $\Delta (\beta_n):= (d_n\beta_n)$
  (где $d_n:=a_n).$
 Для наших целей, удобно переписать указанную факторизацию следующим образом:

  \begin{equation}\label{Eq2}
  T:\, X\overset{W}\to l_\infty\overset{\Delta_1}\to l_2
     \overset{\Delta_0}\to l_2
\overset{\Delta_2}\to  l_1\overset{V}\to Y,
  \end{equation}
где $\Delta_1:=(\sqrt{n^{r/s-1}d_n^r}),$   $\Delta_2:=(\sqrt{n^{r/s-1}d_n^r})$ и
$\Delta_0:=(n^{1-r/s}d_n^{1-r}).$

Предположим, что $\varepsilon>0$ и в факторизации (\ref{Eq1}) $||V||=||W||=1$ и $||(d_n)||_{l_{s,r}}\le (1+\varepsilon) \nu_{s,r}(T).$
Тогда
\begin{eqnarray}
||\Delta_2|| & = & ||\Delta_1||\le \pi_2(\Delta_1)\le ||\sqrt{n^{r/s-1}d_n^r})||_{l_2}  \label{Eq3} \\
{} & = & ||{n^{r/s-1}d_n^r})||^{1/2}_{l_1}\le [(1+\varepsilon) \nu_{s,r}(T)]^{r/2}. \nonumber
  \end{eqnarray}
Также
$\Delta_0\in S_{q,v}(l_2),$ где $1/q=1/s-1$ и $1/v=1/r-1.$            
Более того, так как $1/q-1/v=1/s-1/r,$  $1-r=r/v$ и
${v}/{q}-1+v-{vr}/{s}= v(1/s-1/r+r(1/r-{1}/{s}))=
v(1-r)(1/s-1/r)=r(1/s-1/r),
$
то
    \begin{eqnarray}
 \sigma_{q,v}(\Delta_0) & = & \left(\sum n^{v/q-1}[n^{1-r/s}d_n^{1-r}]^v\right)^{1/v}  \label{Eq4} \\ 
 {} & = &  \left(\sum [n^{1/s-1/r} d_n]^{r}\right)^{1/v}\le [(1+\varepsilon) \nu_{s,r}(T)]^{r/v}.  \nonumber
   \end{eqnarray}
  Факторизация $(s,r)$-ядерного оператора $T,$ описанна в (\ref{Eq2})--(\ref{Eq4}) будет называться
  {\it  $\varepsilon$-допустимой факторизацией для $T.$ }


 Для нас очень важным будут классы $S_{p,q}$ (классы Лоренца-Шаттена) операторов в гильбертовых пространствах,
 представляющие собой обобщения хорошо известных классов Шаттена $S_p.$
Класс  $S_{p,q},$ $0<p,q<\infty,$ рассмотренный впервые Трибелем \cite{Tri67},   
определяется следующим образом. Пусть
$U$ --- компактный оператор в гильбертовом пространстве $H$ и $(\mu_n)$ ---
последовательность его сингулярных чисел  (см., например, \cite{EigPie}, 2.1.13). 
Оператор  $U$ принадлежит пространству $S_{p,q}(H),$
если $(\mu_n)\in l_{p,q}.$
(см., например, \cite{EigPie}, 2.11.15).  
Пространство
$S_{p,q}(H)$ имеет естественную квазинорму
$$\sigma_{p,q}(U)=||(\mu_n)||_{p,q}=\left(\sum_{n=1}^\infty n^{(q/p)-1} \mu_n^q\right)^{1/q}.$$
При $p=q$ класс $S_{p,p}$ совпадает с классом $S_p$ (с квазинормой $\sigma_p).$
 Отметим, что для $p,q\in (0,1]$ выполняется равенство
$N_{p,q}(H)=S_{p,q}(H)$ (см. например, \cite{HiPi}).  
Имеют место включения
$S_{p,q}\subset S_{p',q'},$ если $0<p<\infty$ и $0<q\le q'<\infty$ или
$ 0<p<p'<\infty, o<q,q'<\infty$
(см. \cite{Tri67}, Lemma 2) и   
$$
S_{p,q}\circ S_{p',q'}\subset S_{s,r},\ 1/p+ 1/p'=1/s,\, 1/q+1/q'=1/r.
$$
При этом, если $V\in S_{p,q}$ и $U\in S_{p',q'},$ то
$\sigma_{s,r}(UV)\le 2^{1/s}\sigma_{p',q'}(U)\, \sigma_{p,q}(V)$
(см. \cite{Pie80}, p. 155).   

В случае, когда $p=q, p'=q',$ вместо $2^{1/s}$ в последнем неравенстве можно поставить $1$
\cite{Horn}, \cite{EigPie}, p. 128, \cite{BirSol}, p.262.

Примерами $S_{p,q}$-операторов могут служить диагональные операторы $D$
в $l_2$ с диагоналями $(d_n)$ из $l_{p,q};$ в этих случаях мы пишем $D=(d_n).$

Ниже мы используем понятие 2-абсолютно суммирующей нормы $\pi_2$ для операторов
в банаховых пространствах (см. \cite{PiOP}).  
Отметим, что $\pi_2=\sigma_2$ для операторов в гильбертовых пространствах
(см. \cite{PiOP}).  

\section{Основные результаты}

    \begin{definition}
Оператор
 $T: X\to Y$ {\it факторизуется
через оператор}\,  из $S_{p,q}(H)$ ({\it через $S_{p,q}$-оператор}),
если существуют такие операторы
$A\in L(X, H),$ $U\in S_{p,q}(H)$ и $B\in L(H, Y),$ что $T=BUA.$
Если  $T$  факторизуется
через оператор  из $S_{p,q}(H),$ то  полагаем
$\gamma_{S_{p,q}}(T)= \inf ||A||\, \sigma_{p,q}(U)\, ||B||,$
где инфимум берется по всем возможным факторизациям оператора $T$
через оператор  из $S_{p,q}(H).$
     \end{definition}

Ниже нам понадобится следующий факт.

\begin{proposition}
\label{prop1}
Если оператор $T:X\to Y$ факторизуется через $S_{p,q}$-оператор, то
для любого $\varepsilon>0$ факторизацию $T=BUA,$ где
$A\in L(X, H),\, U\in S_{p,q}(H)$ и $B\in L(H, Y),$
можно выбрать таким образом, что оператор
$B$ инъективен, $\overline{B(H)}= \overline{T(X)}$ и
$||A||\, \sigma_{p,q}(U)\, ||B||\le (1+\varepsilon) \gamma_{S_{pq}}(T).$
\end{proposition}

\begin{proof}
Это простое упражнение. В любом случае, доказательство соответствующего
факта в \cite{ReiMZ} о факторизации через $S_p$ оператор
переносится и на этот случай.
\end{proof}

  \begin{cor}
 \label{cor0}
Пусть $0<p\le1, 0<t\le q\le p.$ В условиях предложения \ref{prop1}, если оператор $T$ конечномерен,
 то
$\gamma_{S_{p,t}}(T)\le (\dim T(X))^{1/t-1/q}\, \gamma_{S_{p,q}}(T).$
\end{cor}

\begin{proof}
Если некоторый оператор $V: X\to Y$ факторизуется через $S_{s,r}$-опе\-ра\-тор, то ассоциированный с ним
оператор $V_0: X\to\overline{V(X)}$ факторизуется через $S_{s,r}$-опе\-ра\-тор с той же $S_{s,r}$-факторизационной квазинормой (предложение 1).
У нас пространство $T(X)$ конечномерно. Поэтому, применяя  к оператору $T$ теорему 1 и предложение 1, мы получаем соответствующую факторизацию $T=BUA$
нашего оператора через $S_{s,r}$-оператор в конечномерном гильбертовом пространстве (размерности $N:=\dim T(X)$).
Если $t\in (0,q],$ то наше утверждение следует из
соответствующих неравенств для $S_{p,t}$-ква\-зинорм в конечномерной ситуации:
если $(\mu_k(U))_{k=1}^N$ --- сингулярные числа оператора $U,$ то, по неравенству Гельдера,
$$\left(\sum_{k=1}^N k^{t/p-1}\mu_k(U)^t\right)^{1/t}\le N^{1/t-1/p}\, \left(\sum_{k=1}^N k^{q/p-1}\mu_k(U)^q\right)^{1/q}.$$
\end{proof}

    Имеет место следующая 
    
        \begin{theorem}
 \label{th1}
  Пусть $m\in \mathbb N.$
Если $X_1, X_2, \dots, X_{m+1}$ --- банаховы пространства,
$0<r_k\le s_k\le1$
 $T_k\in N_{s_k,r_k}(X_k, X_{k+1})$ для
$k=1, 2, \dots, m,$ то произведение
$T:=T_m T_{m-1}\cdots T_1$ может быть факторизовано через оператор из
$S_{s,r}(H),$
где
$1/s= 1/s_1 + 1/s_2 + \dots + 1/s_m - (m+1)/2$ и
$1/r= 1/r_1 + 1/r_2 + \dots + 1/r_m - (m+1)/2.$
Более того,
$$\gamma_{S_{s,r}}(T)\le 2^{1/s} \tilde c\, \prod_{k=1}^m \nu_{s_k,r_k}(T_k),$$
где $\tilde c$ ---
некоторая постоянная $\tilde c:=c_{m; s_1,s_2,\dots, s_{m-1}},$
зависящая только от значений указанных параметров.
Если $s=r,$ то постоянная перед произведением равна единице.
\end{theorem}

\begin{proof}
 Рассмотрим отдельно два случая.

Случай 1: $m>1.$
Для кажого $T_k,$ пусть
$$T_k:=V_kD^{(k)}_2D^{(k)}_0D^{(k)}_1W_k$$
--- его  $\varepsilon$-допустимая факторизация (так что $1/q_k=1/s_k-1$ и $1/v_k=1/r_k-1).$
Отщепим часть призведения $T,$ а именно,  рассмотрим оператор
$$D_0^{(m)}D_1^{(m)}W_m V_{m-1}D_2^{(m-1)} D_0^{(m-1)}D_1^{(m-1)}W_{m-1}\dots
 V_1D_2^{(1)} D_0^{(1)}:\, l_2\to l_2.
$$
Каждый кусок вида $U_{k-1}:=D_1^{(k)}W_k V_{k-1}D_2^{(k-1)} D_0^{(k-1)}$ этого произведения $(k>1),$
$$
U_{k-1}:\, l_2 \overset{D_0^{(k-1)}}\to l_2 \overset{D_2^{(k-1)}}\to l_1 \overset{V_{k-1}}\to X_{k}
\overset{W_k}\to l_\infty \overset{D_1^{(k)}}\to l_2,
$$
есть композиция операторов, для которых
$$\sigma_2(D_1^{(k)}W_k V_{k-1}D_2^{(k-1)})= \pi_2(D_1^{(k)}W_k V_{k-1}D_2^{(k-1)})\le 
||n^{r_k/s_k-1}(d_n^{(k)})^{r_k})||^{1/2}_{l_1}\, ||D_2^{(k-1)}||\le $$
$$\le ||n^{r_k/s_k-1}(d_n^{(k)})^{r_k})||^{1/2}_{l_1}\,
  ||n^{r_{k-1}/s_{k-1}-1}(d_n^{(k-1)})^{r_{k-1}})||^{1/2}_{l_1}\,$$
         $$\le [(1+\varepsilon) \nu_{s_k,r_k}(T_k)]^{r_k/2}\, [(1+\varepsilon) \nu_{s_{k-1},r_{k-1}}(T_{k-1})]^{r_{k-1}/2}$$  
и
$$\sigma_{q_{k-1}, v_{k-1}}(D_0^{(k-1)}) =$$
$$ \left(\sum [n^{1/s_{k-1}-1/r_{k-1}} d_n^{(k-1)}]^{r_{k-1}}\right)^{1/v_{k-1}}
   \le [(1+\varepsilon) \nu_{s_{k-1},r_{k-1}}(T_{k-1})]^{r_{k-1}/v_{k-1}}.                  
$$
Следовательно,
$U_{k-1}\in S_{u_{k-1},w_{k-1}}(l_2),$ где $1/u_{k-1}=1/2+ 1/q_{k-1},$ $1/w_{k-1}=1/2+ 1/v_{k-1},$ и
$$
\sigma_{u_{k-1},w_{k-1}} (U_{k-1}) \le
     2^{1/{u_{k-1}}}
[(1+\varepsilon) \nu_{s_k,r_k}(T_k)]^{r_k/2}\,
[(1+\varepsilon) \nu_{s_{k-1},r_{k-1}}(T_{k-1})]^{1-r_{k-1}/2}\,
$$
Теперь, $T= V_mD_2^{(m)} D_0^{(m)} U_{m-1} U_{m-2} \dots U_1 D_1^{(1)}W_1.$
Здесь
$$
||V_m||=||W_1||=1,\, ||D_2^{(m)}||\le    
  ||{n^{r_m/s_m-1}(d_n^{(m)})^r_m})||^{1/2}_{l_1}\le [(1+\varepsilon) \nu_{s_m,r_m}(T_m)]^{r_m/2},$$
$$\sigma_{q_m,v_m}(D_0^{(m)})=
\left(\sum [n^{1/s_m-1/r_m} d_n{(m)}]^{r_m}\right)^{1/v_m}\le [(1+\varepsilon) \nu_{s_m,r_m}(T_m)]^{r_m/v_m},
$$
$$
||D_1^{(1)}||\le ||{n^{r_1/s_1-1}(d_n^{(1)})^r_1})||^{1/2}_{l_1}\le [(1+\varepsilon) \nu_{s_1,r_1}(T_1)]^{r_1/2}.
$$
Для произведения $ U_{m-1} U_{m-2} \dots U_1,$ имеем:
$U_{m-1} U_{m-2} \dots U_1\in S_{u,w}(l_2),$ где
$1/u= 1/u_{m-1}+\dots +1/u_1= 1/s_{m-1}+\dots +1/s_1- (m-1)/2$
и $1/w= 1/w_{m-1}+\dots +1/w_1= 1/r_{m-1}+\dots +1/r_1- (m-1)/2.$
Более того, для некоторой постоянной $\tilde c:=c_{m; s_1,s_2,\dots, s_{m-1}},$
зависящей только от значений указанных параметров,
$$\sigma_{u,w}(U_{m-1} U_{m-2} \dots U_1)  \le
 \tilde c\,
[(1+\varepsilon) \nu_{s_m,r_m}(T_m)]^{r_m/2}\,
[(1+\varepsilon) \nu_{s_{m-1},r_{m-1}}(T_{m-1})]^{1-r_{m-1}/2}\,
   \times
   $$
 $$ \times
   [(1+\varepsilon) \nu_{s_{m-1},r_{m-1}}(T_{m-1})]^{r_{m-1}/2}\,
[(1+\varepsilon) \nu_{s_{m-2},r_{m-2}}(T_{m-2})]^{1-r_{m-2}/2}\,
\cdots $$
$$ [(1+\varepsilon) \nu_{s_{2},r_{2}}(T_{2})]^{1-r_{2}/2}\,
[(1+\varepsilon) \nu_{s_2,r_2}(T_2)]^{r_2/2}\,
[(1+\varepsilon) \nu_{s_{1},r_{1}}(T_{1})]^{1-r_{1}/2}.
$$
Положим
$$
A=D_1^{(1)}W_1,\, B=V_m D_2^{m}, \, U= 
D_0^{m} U_{m-1} U_{m-2} \dots U_1. 
$$
Тогда $T=BUA$ and, by (\ref{Eq3})--(\ref{Eq4}),
$$
||A||\le [(1+\varepsilon) \nu_{s_1,r_1}(T_1)]^{r_1/2},
||B||\le [(1+\varepsilon) \nu_{s_m,r_m}(T_m)]^{r_m/2},
$$
$$\sigma_{s,r}(U)\le 2^{1/s}
\sigma_{q_m,v_m}(D_0^{m})\, \sigma_{u,w}(U_{m-1} U_{m-2} \dots U_1)
\le$$
$$\le 2^{1/s} \tilde c\, [(1+\varepsilon)\nu_{s_m,r_m}(T_m)]^{1-r_m} [(1+\varepsilon)\nu_{s_m,r_m}(T_m)]^{r_m/2}
[(1+\varepsilon) \nu_{s_{m-1},r_{m-1}}(T_{m-1})] \times
$$
$$ [(1+\varepsilon) \nu_{s_{m-2},r_{m-2}}(T_{m-2})] \dots
[(1+\varepsilon) \nu_{s_2,r_2}(T_2)] [(1+\varepsilon) \nu_{s_{1},r_{1}}(T_{1})]^{1-r_{1}/2}
$$
(напомним, что $r_m/v_m=1-r_m,$ $1/q_m= 1/s_m-1$ and $1/v_m=1/r_m-1).$
Следовательно,
$$
\gamma_{S_{s,r}(T)}\le
2^{1/s} \tilde c\,\, \prod_{k=1}^m \nu_{s_k,r_k}(T_k).
$$


Случай 2: $m=1.$
Пусть $0<r\le s<1$ или $o<r<s=1$ (ситуация, в которой $s=r=1$ рассмотрена в \cite{ReiMZ}).
В этом случае 
 $1/q_1=1/s_1-1$ и $1/v_1=1/r_1-1)$ (и, следовательно, $r_1/v_1=1-r_1$)
Для $\varepsilon$-допустимой факторизации
$T_1:=V_1D^{(1)}_2D^{(1)}_0D^{(1)}_1W_1$ имеем:
$$
||V_1 D^{(1)}_2||\, ||D^{(1)}_1W_1||\, \sigma_{q_1,v_1}(D^{(1)}_0) \le
$$
$$
   [(1+\varepsilon) \nu_{s_1,r_1}(T_1)]^{r_1} \, [(1+\varepsilon) \nu_{s_1,r_1}(T_1)]^{r_1/v_1}=
(1+\varepsilon) \nu_{s_1,r_1}(T_1).
$$
Следовательно,
$$
\gamma_{S_{s_1,r_1}(T_1)}\le
 \nu_{s_1,r_1}(T_1).
$$
\end{proof}

Из доказательства теоремы получаем даже более сильное утверждение.

 \begin{cor}
 \label{cor1}
В условиях теоремы \ref{th1}, для любого $\delta>0$ произведение
$T:=T_m T_{m-1}\cdots T_1$ может быть факторизовано следующим образом
$$
T:\, X_1\overset{\tilde A}\to l_2\overset{\tilde U}\to l_2\overset{\tilde B}\to X_m,
$$
где $\pi_2(\tilde A)\le 1, 
$
$\pi_2(\tilde B^*)\le 1 
$ и
$$\sigma_{s,r}(\tilde U)\le 
(1+\delta) 2^{1/s} \tilde c \prod_{k=1}^{m} \nu_{s_k,r_k}(T_k).
$$
\end{cor}

\begin{proof}
Рассмотрим операторы $A, U, B$ из доказательства теоремы \ref{th1} и положим
$$\tilde A:= [(1+\varepsilon)\nu_{s_1,r_1}(T_1)]^{-r_1/2} A,\
\tilde B:=[(1+\varepsilon)\nu_{s_m,r_m}(T_m)]^{-r_1/2} B,$$
$$\tilde U:= (1+\varepsilon)^{r_1/2+r_m/2} \nu_{s_1,r_1}(T_1)^{r_1/2}  \nu_{s_m,r_m}(T_m)^{r_m/2}U.$$
Выбрав достаточно малое $\varepsilon=\varepsilon(\delta),$ получим желаемую факторизацию.

\end{proof}

\begin{remark}
\label{rem1}
Если $s=r$ (тогда все $s_j$ и, соответственно,  все $r_j$ равны между собой),
постоянная в неравенстве из теоремы 1 (соотв., в Следствии \ref{cor1}) равна единице (соотв., $1+\delta).$) \cite{ReiMZ}.
Действительно, в этом случае постоянные в неравенствах Гельдера для соотношений типа $S_r\subset S_p\circ S_q$
равны единице.
\end{remark}

 \begin{cor}
 \label{cor2}
В условиях теоремы \ref{th1}, пусть $X_1=X_m$
и $\delta>0.$
Последовательность $(\lambda_n(T))$ собственных чисел оператра $T$ лежит в
пространстве $S_{\tilde s, \tilde r,}$ где
$\tilde s=1/2+1/s,$ $\tilde r=1/2+1/r.$ При этом
$$
||(\lambda_n(T))||_{\tilde s, \tilde r}\le 2^{1/s+1/{\tilde s}} \tilde c \prod_{k=1}^{m} \nu_{s_k,r_k}(T_k).
$$
Если $s=r,$ то постоянная справа в этом неравенстве равна 1.
\end{cor}

\begin{proof}
Пусть $\delta>0.$
В обозначениях следствия \ref{cor1} и теоремы \ref{th1}, рассмотрим следующую диаграмму:
$$
\tilde AT:\, X_1\overset{\tilde A}\to l_2\overset{\tilde U}\to l_2\overset{\tilde B}\to X_1\overset{\tilde A}\to l_2.
$$
Последовательность собственных чисел $(\lambda_n(T))$ оператора $T$ совпадает (с учетом их алгебраических кратностей)
с полной последовательностью собственных чисел оператора $\tilde A\tilde U\tilde B$
(см., например, \cite{PiOP}).  
Так как
$$\tilde A\tilde U\tilde B\in S_2\circ S_{s,r}\subset S_{\tilde s, \tilde r}$$
и $r\le s,$
то, по неравенству Вейля (см. \cite{Kon86}, 1.c.13),   
 $$||(\lambda_n(T))||_{\tilde s, \tilde r}\le \sigma_{\tilde s, \tilde r}(\tilde A\tilde U\tilde B)\le
2^{1/{\tilde s}}\pi_2{(\tilde A)} \sigma_{s,r}(\tilde U)\le
 (1+\delta) 2^{1/s+1/{\tilde s}} \tilde c \prod_{k=1}^{m} \nu_{s_k,r_k}(T_k).
 $$
В силу произвольности $\delta,$ наше утверждение доказано.
\end{proof}

\begin{remark}
\label{rem2}
Имеет место более сильный вариант последнего следствия.
Пусть $\Sigma_{p,q}$ есть пространство всех неупорядоченных комплексных последовательностей
$\alpha=(\alpha_k),$ для которых конечна величина $\rho_{p,q}(\alpha,\beta):= \inf \text{dist}_{p,q}(\alpha_k-\beta_k),$ где
$\text{dist}_{p,q}$ --- метрика на пространстве $l_{p,q},$ порождающая его естественную топологию
(см., например, [\cite{EigPie}, 6.1-6.2, а также \cite{Groth}, Chap. 2, pp. 20-21).
Здесь infimum берется по всевозможным последовательностям $(\alpha_k)$  (соответственно, $(\beta_k)$) из $l_{p,q},$
которые определяют неупорядоченную последовательность $\alpha$  (соответственно, $(\beta_k)$).
Тогда, в условияях теоремы  \ref{th1} (и следствия  \ref{cor2}),
{\it естественное отображение
 $$N_{s_m,r_m}\circ N_{s_{m-1},r_{m-1}} \cdots \circ N_{s_1,r_1}\to \Sigma_{\tilde s, \tilde r}$$
непрерывно}.
Доазательство сводится к случаю гильбертова пространства, в котором
непрерывность естественного отображения $S_{\tilde s, \tilde r}\to \Sigma_{\tilde s, \tilde r}$
получается, например, с помощью рассужений, аналогичных тем, что приведены
в \cite{BirSol}, Chap 11, \S7.  
Впрочем, этот фкт для $s$-ядерных опраторов и операторов из класса $S_p$
был известен еще А. Гротендику \cite{Groth},Chap. 2, pp. 20-21.
\end{remark}

\begin{remark}
\label{rem3}
Следствие \ref{cor2} точно для случая, когда $s=r$ (см. \cite{ReiMZ}). В общем случае
точным оказывается такой результат.
\end{remark}

\begin{proposition}
\label{prop0}
  Пусть $m\in \mathbb N.$
Если $X_1, X_2, \dots, X_{m+1}$ --- банаховы пространства,
$0<r_k\le s_k\le1$
 $T_k\in N_{s_k,r_k}(X_k, X_{k+1})$ для
$k=1, 2, \dots, m,$ то собственные числа произведения
$T:=T_m T_{m-1}\cdots T_1$ 
лежат в пространстве $l_{p,q},$ где
$1/p= 1/s_1 + 1/s_2 + \dots + 1/s_m - m/2$
$1/q= \sum_{k=1}^m 1/r_k.$
\end{proposition}

\begin{proof}
Мы воспользуемся теми фактами, что идеал операторов Вейля $\mathfrak L^{(x)}_{p,q}$ \cite{EigPie} типа $l_{p,q}$  
имеет спектральный тип $l_{p,q}$ (т. е. последовательности собственные чисел операторов
Вейля лежат в $l_{p,q};$ см. \cite{EigPie}, 3.6.2) и идеал $(p_0,q)$-ядерных операторов, где    
$1/p_0=1/2+1/p,$ 
вложен в этот идеал операторов Вейля \cite{HiPi}, p. 243.    
Из этих фактов следует, что оператор $T$ лежит в соответсвующем произведении идеалов операторов
Вейля. По теореме о произведениях \cite{EigPie}, 2.4.18,
$$\mathfrak L^{(x)}_{p_1,q_1}\circ \mathfrak L^{(x)}_{p_2,q_2}\subset \mathfrak L^{(x)}_{p,q},$$
где $1/p=1/p_1+1/p_2,\, 1/q=1/q_1+1/q_2.$
Поэтому $(\lambda_n(T))\in \mathfrak L^{(x)}_{\tilde s,q},$ где 
и $1/q= \sum_{k=1}^m 1/r_k.$
\end{proof}

В работе \cite{ReiMZ}, как видно из доказательства теоремы 3 той работы, мы на самом деле, кроме всего прочего, 
установили следующий результат, показывающий, что утверждение следствия \ref{cor2}
точно для случая, когда рассматриваются $p$-ядерные операторы и классы Шаттена $S_p$
(т. е. в следствии \ref{cor2} $s=r$). 
Сформулируем результат в полной общности (используя и заключение теоремы 1 при $s=r$).

\begin{theorem}
\label{th0}
В условиях теоремы 1, пусть $\lambda:=(\lambda_k(T))$ есть последовательность всех собственных чисел
оператора $T,$ взятых с учетом кратностей. 
Если $s_k=r_k, k=1,2, \dots, m,$ то  $\lambda\in l_q,$ где $1/q= 1/r_1 + 1/r_2 + \dots + 1/r_m - m/2,$
причем $(\sum_{k=1}^\infty |\lambda_k|^q)^{1/q}\le \prod_{k=1}^m \nu_{r_k}(T_k).$
Неравенство неулучшаемо с точностью до абсолютной постоянной
(для любого количества операторов и для любого набора 
чисел $0<r_k= s_k\le1.)$
\end{theorem}

\section{Факторизация операторов из $N_{s;2}$}  

В этом разделе мы рассмотрим задачу о факторизации через операторы из классов
Лоренца-Шаттена операторов типа $N_{s;2}$ (в индексе точка с запятой!).
Здесь $0<s\le2.$
Для банаховых пространств $X,Y$
пространство $N_{s;2}(X,Y)$ состоит из ядерных операторов $T: X\to Y,$ которые представимы в виде
$ Tx=\sum_{n=1}^\infty a_n \langle x'_n,x\rangle  y_n,$
для $x\in X,$ где
$$
 ||x'_n||\le1, (a_n)\in l_{s},
||(y_n)||_{2}^{\text{weak}}:=\sup_{||y'||\le1} \left(\sum_{n=1}^\infty |\langle y',y_n\rangle|^2\right)^{1/2}<\infty.
$$
Такие операторы будем называть $(s;2)$-ядерными.
Отметим, то мы можем (и будем) предполагать, что $||x'_n||= 1,$
 $||(y_n)||_{2}^{\text{weak}}=1$ для всех $n$ и что
последовательность $(a_n)$  неотрицательная и убывающая..
Мы используем обозначение
 $\nu_{(s;2)}(T)$ для естественной квазинормы
$\inf   ||(a_n)||_{l_{s}}.$ Идеалы $(s;2)$-операторов являются частными  случаями идеалов
$(s,r,q)$-ядерных операторов из \cite{PiOP}, 18.1.  

  Каждый $(s;2)$-ядерный оператор $T: X\to Y$ допускает факторизацию следующего вида:
 \begin{equation}\label{Eq5}
  T:\, X\overset{W}\to l_\infty\overset{\Delta}\to l_2\overset{V}\to Y,
   \end{equation}
  где $||V||=||W||=1$ и $\Delta$ --- диагональный оператор с диагональю $(d_n)\in l_{s}.$
  Действительно, достаточно положить
  $Wx:= (\langle x'_k,x\rangle),$ $V(\alpha_n):= \sum \alpha_n y_n$ и $\Delta (\beta_n):= (d_n\beta_n)$
  (где $d_n:=a_n).$
 Для наших целей, удобно переписать указанную факторизацию следующим образом
 (мы следуем идеям из\cite{EigPie}, 3.8.6):   

  \begin{equation}\label{Eq6}
  T:\, X\overset{W}\to l_\infty\overset{\Delta_1}\to l_2
     \overset{\Delta_0}\to l_2
  \overset{V}\to Y,
  \end{equation}
где $\Delta_1:=({d_n}^{s/2})$    и
$\Delta_0:=(d_n^{s/q}),$ где $1/q=1/s-1/2.$
Предположим, что $\varepsilon>0$ и в факторизации (\ref{Eq6}) $||V||=||W||=1$ и $||(d_n)||_{l_{s}}\le (1+\varepsilon) \nu_{s;2}(T).$
Тогда
\begin{equation}\label{Eq7}
\pi_2(\Delta_1)\le ||(d_n^{s/2})||_{l_2}=
  ||(d_n)||^{s/2}_{l_s}\le [(1+\varepsilon) \nu_{s;2}(T)]^{s/2}.
  \end{equation}
Также
$\Delta_0\in S_{q}(l_2).$ 
и $\sigma_q(\Delta_0)\le  ||(d_n)||^{s/q}_{l_s}\le [(1+\varepsilon) \nu_{s,2}(T)]^{s/q}.$
Поэтому
$T= V\Delta_0\Delta_1W\in S_q\circ \Pi_2$ и
$||T||_{S_q\circ \Pi_2}\le  \nu_{s;2}(T).$

Теперь нетрудно получить следующий результат.

\begin{theorem}
 \label{th3}
  Пусть $m\in \mathbb N.$
Если $X_1, X_2, \dots, X_{m+1}$ --- банаховы пространства,
$0< s_k\le2$                           
и
 $T_k\in N_{s_k;2}(X_k, X_{k+1})$ для
$k=1, 2, \dots, m,$ то произведение
$T:=T_m T_{m-1}\cdots T_1$ может быть факторизовано через оператор из
$S_{s}(H),$
где
$1/s= \sum_{k=1}^m 1/s_k-1/2.$
   Более того,
$$\gamma_{S_{s}}(T)\le  \prod_{k=1}^m \nu_{s_k;2}(T_k),$$
\end{theorem}

\begin{proof}
Следуя предыдущим рассуждениям, факторизуем каждый из операторов $T_k$
как произведение $V^{k}\Delta_0^{k}\Delta_1^{k}W^{k}$
Тогда на "стыке" двух операторов появится оператор вида $\Delta_1^{k+1}W^{k+1}V^k,$
у которого $\sigma_2$-норма не превосходит $\pi_2$-нормы оператора $\Delta_1^{k+1},$
т. е. $[(1+\varepsilon)\nu_{s_{k+1};2}]^{s_{k+1}/2}.$
За ним следует оператор $\Delta_0^{k+1},$ для которого $\sigma_{q_{k+1}}(\Delta_0^{k+1})\le [(1+\varepsilon)\nu_{s_{k+1};2}]^{s_{k+1}/q_{k+1}}.$
Поэтому
$$T\in L\circ S_{q_m}\circ S_2\circ S_{q_{m-1}}\circ S_2\circ \cdots \circ S_{q_2}\circ S_2\circ S_{q_1}\circ \Pi_2\circ L.$$
Здесь слева и справа в получаемой факторизации $T$ появляются операторы $V^m$ и $\Delta_1^1W^1$
соответственно.
Между ними находятся произведения вида $S_{q_k}\circ S_2$ ($m-1$ штука). Особняком
входит в произведение оператор из $S_{q_1}.$
Таким образом, используя неравенство Гельдера для произведений
операторов из классов Шаттена, получаем:
$T\in L\circ S_{s}\circ \Pi_2,\ \gamma_{S_{s}}\le \prod_{k=1}^m \nu_{s_k;2}(T_k).$
\end{proof}

Важно отметить, что в полученном неравенстве постоянная оценки справа (= 1) не зависит
от параметров.

\begin{cor}
 \label{cor3}
В услвиях теоремы \ref{th3}, для любого $\delta>0$ оператор $T$ представим
в виде произведения $BUA\in L\circ S_{s}\circ \Pi_2,$ где
$$||B||=1,\, \sigma_s(U)\le (1+\delta) \prod_{k=1}^{m} \nu_{s_k;2}(T_k)
\text{ и } \pi_2(A)=1.
$$
\end{cor}

\begin{proof}
Это вытекает непосредственно из доказательства теоремы \ref{th3}.
\end{proof}

\begin{cor}
 \label{cor4}
В условиях теоремы \ref{th3}, пусть $X_1=X_m$
и $\delta>0.$
Последовательность $(\lambda_n(T))$ собственных чисел оператора $T$ лежит в
пространстве $S_{\tilde s}$ где
$\tilde s=\sum_{k=1}^m 1/s_k.$  При этом
$$
||(\lambda_n(T))||_{\tilde s}\le \prod_{k=1}^{m} \nu_{s_k;2}(T_k).
$$
\end{cor}

\begin{proof}
Применяем предыдущее следствие. Так как наборы собственных чисел операторов
$T=BUA: X_1\overset{A}\to l_2\overset{U}\to l_2\overset{B}\to X_1$ и
$ABU: l_2\overset{U}\to l_2\overset{B}\to X_1 \overset{A}\to l_2$ совпадают
(вместе с кратностями) и $\Pi_2(l_2)=S_2(l_2)$ изометрично, то
 $(\lambda_n(T))\in l_{\tilde s}.$ Неравенство следует из неравенства Гельдера
 для произведений $S_p$-операторов и из неравенства Вейля между
 $l_{p}$-квазинормами последовательностей собственных и сингулярных чисел.
\end{proof}

Приведем конечномерные варианты теоремы \ref{th3} и следствия \ref{cor4}.

\begin{cor}
 \label{cor5}
Пусть $m\in \mathbb N,$
$X_1, X_2, \dots, X_{m+1}$ --- банаховы пространства,
$0< s_k\le2$
и
 $T_k\in N_{s_k,2}(X_k, X_{k+1})$ для
$k=1, 2, \dots, m.$ Пусть, далее, $1/s= \sum_{k=1}^m 1/s_k-1/2$ и
$0<t\le s.$
Если оператор
$T:=T_m T_{m-1}\cdots T_1$ конечномерен, то
$$\gamma_{S_{s}}(T)\le
(\dim T(X_1))^{1/t-1/s}\,
\prod_{k=1}^m \nu_{s_k;2}(T_k),$$
\end{cor}

\begin{proof}
Достаточно применить следствие \ref{cor0}.
\end{proof}

\begin{cor}
 \label{cor6}
Пусть $m\in \mathbb N,$
$X_1, X_2, \dots, X_{m+1}$ --- банаховы пространства,
$0< s_k\le2$
и
 $T_k\in N_{s_k;2}(X_k, X_{k+1})$ для
$k=1, 2, \dots, m.$ Пусть, далее, $1/s= \sum_{k=1}^m 1/s_k-1/2$ и
$0<t\le s.$
Если оператор
$T:=T_m T_{m-1}\cdots T_1$ конечномерен, то
$$
||(\lambda_n(T))||_{\tilde s}\le
(\dim T(X_1))^{1/t-1/\tilde s}\,
\prod_{k=1}^{m} \nu_{s_k;2}(T_k).
$$
 где
$1/\tilde s=\sum_{k=1}^m 1/s_k.$
\end{cor}

\begin{proof}
Применяем предыдущее следствие и рассуждения из
доказательства следствия \ref{cor4}.
\end{proof}

Результаты, полученные в последних двух следствиях, не улучшаемы (с точностью до абсолютной постоянной
в неравенстве). По-существу, соответствующий пример имеется в работе \cite{ReiMZ}, теорема 2.
Удобно сформулировать в виде отдельного утверждения результат, установленный в доказательстве
теоремы 2 в той статье (мы изменим здесь обозначения параметров из \cite{ReiMZ} для согласования с
нашими обозначениями).

 \begin{proposition}
 \label{prop2}
Существует постоянна $G>0$ такая, что для любого натурального числа $n$
 найдется оператор $A_n: l_1^n\to l_1^n,$ обладающий следующим свойством.
Если $m\in \mathbb N,$  $p_k\in (0,1]$  для $k=1, 2, \dots, m,$
$1/p= 1/p_1 + 1/p_2 + \dots + 1/p_m - (m+1)/2$ и $u\in (0, p],$ то
  $$
 \gamma_{S_u}(A^m_n)\ge G n^{1/u-1/p}\,
\prod_{k=1}^m \nu_{p_k}(A_n)=G n^{1/u-1/p}\,
\left(\sum_{\lambda(A_n^m)} |\lambda(A_n^m)|^v\right)^{1/v},  
 $$
 где $(\lambda(A_n^m))$ --- полный набор собственных чисел оператора $A_n^m$ и
 $1/v=1/2+1/p.$
 \end{proposition}


Теперь о точности неравенств из следствий \ref{cor5} и \ref{cor6}.

 \begin{theorem}
 \label{th2}
 Существует постоянна $G>0$ такая, что для любого натурального числа $n$
 найдется оператор $A_n: l_1^n\to l_1^n,$ обладающий следующим свойством.
Если $m\in \mathbb N,$  $s_k\in (0,1]$  для $k=1, 2, \dots, m,$
$1/s= 1/s_1 + 1/s_2 + \dots + 1/s_m - (m+1)/2$ и $t\in (0, r],$ то
  $$
 \gamma_{S_t}(A^m_n)\ge G n^{1/t-1/s}\,
\prod_{k=1}^m \nu_{s_k}(A_n) \ \ \text{ и }\
\prod_{k=1}^m \nu_{s_k;2}(A_n)=
\left(\sum_{\lambda(A_n^m)} |\lambda(A_n^m)|^{\tilde s}\right)^{1/\tilde s}.
 $$

 \end{theorem}

\begin{proof}
Оператор, о котором говорится в предложение \ref{prop2}, порождается унитарной матрицей
$$\left({n^{-1/2}}\, e^{\frac{2\pi jl}n i}\right)\  (j,l=1,2,\dots, n).$$
Рассмотрим несколько новых параметров. Пусть
$1/p_k=1/s_k+1/2,$ $1/p=1/s+1/2,$ $1/u=1/t+1/2,$  $1/q=1/2+1/u=1/t+1$
Для любого $k$ имеют место соотношения
$$\nu_{s_k;2}(A_n)\le \nu_{p_k}(A_n)\le n^{1.s_k}=(\sum_{\lambda(A_n)} |\lambda(A_n)|^{s_k})^{1/s_k}
\le \nu_{s_k;2}(A_n).
$$
Поэтому
$$\prod_{k=1}^m \nu_{s_k;2}(A_n)=n^{\sum 1/s_k}=n^{1/\tilde s}.$$
Кроме того,
$1/u-1/p=1/t-1/s$ и $\gamma_{S_{u}}(A_n^m)\le \gamma_{S_t}(A_n^m)\cdot n^{1/2}$
(неравенство Гельдера).
По предложению \ref{prop2},
 $$
 \gamma_{S_t}(A^m_n)\ge n^{-1/2}G n^{1/t-1/s}\,
\prod_{k=1}^m \nu_{s_k;2}(A_n)= n^{1/2}G n^{1/t-1/s}\,
\left(\sum_{\lambda(A_n^m)} |\lambda(A_n^m)|^{\tilde s}\right)^{1/\tilde s}.
 $$
\end{proof}

В работе \cite{ReiMZ} мы использовали предложение \ref{prop2} для построения
примера произведения $N_{s}$-операторов в $l_1,$ для которого  утверждения следствий \ref{cor2}
и \ref{cor3} точны (при $s=r;$ см. теорему 3 в \cite{ReiMZ}.
В наших последних
утверждениях (следствия \ref{cor5} и \ref{cor6} ) получены оценки факторизационных $\gamma_{S_p}$-квазинорм, а также
$l_p$-квазинорм последовательностей собственных чисел произведений $N_{s;2}$операторов.
Предыдущая теорема
показывает, что эти оценки точны в конечномерных ситуациях.

С другой стороны, в доказательстве теоремы \ref{th2} существенно использовался
тот факт, что для рассматриваемого там оператора $A_n$ его $\nu_{p_k}$- и
$\nu_{s_k;2}$-квази\-нор\-мы совпадают.
Это приводит нас к результату, соответствующему теореме 3 из \cite{ReiMZ}.
Доказательство буквально то же, с заменой обозначений (одной квазинормы на другую).
Основное совпадение рассматриваемых сейчас квазинорм (и это существенно в доказательстве
той теоремы) --- это то, что обе они, как $\nu_{p_k},$ так и $\nu_{s_k;2}$ $(1/p_k=1/2+1/s_k)$
 являются  полными $p_k$-нормами.
  Поэтому доказательство теоремы 3  из \cite{ReiMZ} почти дословно проходит в нашем случае для
  произведений операторов из $N_{s;2},$  приводя к следующей теореме.

\begin{theorem}
 \label{th4}
Пусть $m\in \mathbb N,$
$s_k\in (0,2]$  для
$k=1, 2, \dots, m$ и
$1/s= 1/s_1 + 1/s_2 + \dots + 1/s_m - 1/2.$
Существуют операторы $T_k\in N_{s_k;2}(X_k, X_{k+1})$ в банаховых пространствах
такие, что композиция
$T:=T_m T_{m-1}\cdots T_1$
факторизуется через оператор  из $S_s(H),$
но не факторизуется ни через какой оператор из $S_t(H),$ если $t\in (0,r).$
В качестве всех пространств $X_k$ можно взять пространство $l_1.$
\end{theorem}

Итак, для произведений $(s;2)$-ядерных и для произведений $p$-ядерных операторов,
действующих из  в $L_1$-пространств в $L_1$-пространства,
получены точные ответы на вопросы о факторизации через операторы из классов
Шаттена и о распределении их собственных чисел (но оставаясь в шкалах
$S_p$ и $l_s$).

{{\it Как выглядят ответы на соответствующие вопросы для произведений операторов, которые
(произведения) действуют
в $L_p$-пространствах?}

Мы оставим ответы на этот и другие (более тонкие) вопросы до следующей статьи.
Между прочим, для случая $L_2$-пространств ответы почти очевидны.
Для того, чтобы сформулировать некоторые результаты об операторах в $L_p$-пространствах,
надо "интерполировать" полученные выше утверждения между $L_1$- и $L_2$-случаями.



\begin{thebibliography}{99}

\bibitem{FaRei} %
Рейнов О. И., О произведении ядерных операторов,
Функц. анализ и его прил., том 51, выпуск 4, 2017, 90--91

\bibitem{ReiMZ} %
Рейнов О. И., О произведении s-ядерных операторов, Матем. заметки, том 107, выпуск 2, 2020, 311--316.

\bibitem{BirSol} %
Michael Sh. Birman, M.Z. Solomjak, Spectral Theory of Self-Adjoint Operators in Hilbert Space,
D. Reidel Pub. Co., 1987

\bibitem{Ca16} %
T. Carleman, Uber die Fourierkoeffizienten einer stetigen Funktion:
Aus einem Brief an Herrn A. Wiman. Acta Math. 41 (1916), 377--384.
[Ca16]

 \bibitem{Groth} %
 A. Grothendieck, Produits tensoriels topologiques et espases nucl\'eaires,
  Mem. Amer. Math. Soc., Volume 16,    1955, 196 + 140.
  [Groth]
  

\bibitem{HiPi} %
A. Hinrichs, A. Pietsch, p -nuclear operators in the sense of Grothendieck,
Math. Nachr. 283, No. 2, 232--261 (2010).

\bibitem{Horn} %
Horn, A.  (1950) On the singular values of a product of completely continuous operators. Proc. Nat. Acad.
Sci. USA 36, 374--375.  



\bibitem{WBJ79} %
W.B. Johnson, H. Konig, B. Maurey, J.R. Retherford ?
[1979] Eigenvalues of p-summing and lp-type operators in Banach spaces, J. Funct. Anal.
32, 353--380. 


\bibitem{Kon86} %
Hermann Konig, Eigenvalue Distribution of Compact Operators, Springer Basel AG, 1986.

\bibitem{Lal} %
T. Lalesco, Un theor\'eor\`eme sur les noyaux compos\'es, Bull. Sect. Sci. Acad. Roumaine 3, 1915, 271--272

\bibitem{Pie80} %
A. Pietsch, Weyl numbers and eigenvalues of operators in Banach spaces, Math. Ann. 247, 1980,
149--168.  

\bibitem{EigPie} %
A. Pietsch, Eigenvalues and s-numbers (Cambridge University Press, Cambridge, 1987).

\bibitem{PiOP} %
A. Pietsch, History of Banach Spaces and Linear Operators (Birkhauser, Boston, 2007).

\bibitem{SchatTensors} %
Schatten, R. ,  A theory of cross-spaces. Annals of Math. Studies 26, Princeton Univ. Press 1950. 

\bibitem{Schur} %
I. Schur, Uber die charakteristischen Wurzeln einer linearen Substitution mit einer Anwendung 
auf die Theorie der Integralgleichungen, Math. Ann. 66, 1909, 488--510.

\bibitem{Tri67} %
H. Triebel,  Uber die Verteilung der Approximationszahlen kompakter Operatoren in Sobolev-
Besov-Raumen, Invent. Math. 4, 1967, 275--293.


  
  
  
  
   \end{thebibliography}
\end{document}